\documentclass[11pt]{amsart}

\usepackage[ruled,vlined, titlenotnumbered, algo2e]{algorithm2e}
\usepackage{amsthm}
\usepackage{amsmath,latexsym}
\usepackage{amssymb}
\usepackage{setspace}
\usepackage[english, activeacute]{babel}
\usepackage{float}
\usepackage{hyperref}

\usepackage[all]{xy}

\newcommand{\g}{\mathrm{g}}
\newcommand{\F}{\mathrm{F}}
\newcommand{\G}{\mathrm{G}}
\newcommand{\m}{\mathrm{m}}
\newcommand{\e}{\mathrm{e}}

\newcommand{\T}{\mathrm{T}}
\newcommand{\Ap}{\mathrm{Ap}}
\newcommand{\Z}{\mathbb{Z}}
\newcommand{\N}{\mathbb{N}}

\newcommand{\Nrm}{\mathrm{N}}
\newcommand{\Zrm}{\mathrm{Z}}
\newcommand{\Urm}{\mathrm{U}}

\newcommand{\D}{\mathrm{D}}
\renewcommand{\P}{\mathcal{P}}
\newcommand{\Sem}{\mathcal{S}}

\newcommand{\Ho}{\mathcal{H}}
\newcommand{\I}{\mathcal{I}}
\newcommand{\Kunz}{\mathcal{K}}
\newcommand{\dsum}{\displaystyle\sum}

\newtheorem{theo}{Theorem}
\newtheorem{ex}[theo]{Example}
\newtheorem{lem}[theo]{Lemma}
\newtheorem{cor}[theo]{Corollary}
\newtheorem{prop}[theo]{Proposition}

\begin{document}

\title[Numerical semigroups with fixed Frobenius number]{A semilattice structure for the set of numerical semigroups with fixed Frobenius number}

\author[V. Blanco and J.C. Rosales]{V. Blanco and J.C. Rosales\medskip\\
D\lowercase{epartamento de } \'A\lowercase{lgebra}, U\lowercase{niversidad de } G\lowercase{ranada}\medskip\\
\texttt{\lowercase{vblanco@ugr.es}} \; \texttt{\lowercase{jrosales@ugr.es}}}

\address{Departamento de \'Algebra, Universidad de Granada}

\keywords{Numerical semigroup, Frobenius number, Partitions of sets, Kunz-coordinates vectors.}

\subjclass[2010]{20M14, 11D07, 05A18, 90C10.}


\begin{abstract}
We present a procedure to enumerate the whole set of numerical semigroups with a given Frobenius number $F$, $\Sem(F)$. The methodology is based on the construction of a partition of $\Sem(F)$ by a congruence relation. We identify exactly one irreducible and one homogeneous numerical semigroup at each class in the relation, and from those two elements we reconstruct the whole class. An alternative more efficient method is proposed based on the use of the Kunz-coordinates vectors of the elements in $\Sem(F)$.
\end{abstract}

\maketitle

\section{Introduction}

Let $\N$ be the set of nonnegative integer numbers. A numerical semigroup is a subset $S$ of $\N$ closed under
addition, containing zero and such that $\N\backslash S$ is finite.  The largest integer not belonging to $S$ is called the \textit{Frobenius number} of $S$ and we denote it by $\F(S)$.  The cardinal of its set of gaps, $\G(S)=\N \backslash S$, is usually called the genus of $S$ and it is denoted by $\g(S)$.

Let $F$ be a positive integer and $\Sem(F)$ the set of all the numerical semigroups with Frobenius number $F$. The main goal of this paper is to describe a procedure to enumerate all the elements in $\Sem(F)$.

It is clear that $(\Sem(F), \cap)$ is a semilattice, that is, a commutative semigroup such that all its elements are idempotent. In Section \ref{sec:2} we define a congruence relation $R$ over $\Sem(F)$ verifying that the quotient semilattice $\frac{\Sem(F)}{R} = \{[S]: S \in \Sem(F)\}$ is a partition of $\Sem(F)$ into sets which are closed under unions and intersections and that have maximum and minimum (with respect to the inclusion ordering).

A numerical semigroup is \textit{irreducible} if it cannot be expressed as an intersection of two numerical semigroups containing it properly. This notion was introduced in \cite{pacific} where it is also proven, from \cite{barucci} and \cite{froberg}, that the family of irreducible numerical semigroups is the union of two families of numerical semigroups that have been widely studied and that have special importance in this theory: symmetric and pseudo-symmetric numerical semigroups. We denote by $\I(F)$ the set of irreducible numerical semigroups with Frobenius number $F$. In Section \ref{sec:3} we see that each class in $\frac{\Sem(F)}{R}$ contains an unique irreducible numerical semigroup which is the maximum of that class.

If $A$ is a nonempty subset in $\N$, we denote by $\langle A \rangle$ the submonoid of $(\N, +)$ generated by $A$, that is $\langle A \rangle = \{\lambda_1a_1 + \cdots + \lambda_na_n: n\in \N \backslash \{0\}, a_1, \ldots, a_n \in A, \text{ and } \lambda_1, \ldots, \lambda_n \in \N\}$. It is well-known (see for instance \cite{springer}) that $\langle A \rangle$ is a numerical semigroup if and only if $\gcd(A)=1$. If $S$ is a numerical semigroup and $S=\langle A \rangle$, then we say that $A$ is a system of generators of $S$. If there not exists any other proper subset of $A$ generating $S$, we say that $A$ is a minimal system of generators of $S$. Every numerical semigroup admits an unique minimal system of generators and that such a system is finite. If $S$ is a numerical semigroup, the elements in a minimal system of generators of $S$ are called \textit{minimal generators} of $S$.

We say that a numerical semigroup $S$ is homogeneous if its minimal generators do not belong to the open interval $]\frac{\F(S)}{2}, \F(S)[$. We denote by $\Ho(F)$ the set of homogeneous numerical semigroups with Frobenius number $F$. In Section \ref{sec:3} we see that each class in $\frac{\Sem(F)}{R}$ contains an unique homogeneous numerical semigroup which is the minimum of that class. As a consequence we have that the sets $\frac{\Sem(F)}{R}$, $\I(F)$ and $\Ho(F)$ have the same cardinal.

As a consequence of the above results we have that $\Sem(F) = \bigcup_{S \in \I(F)} [S]$. Therefore, to compute all the elements in $\Sem(F)$ it is enough to compute the elements in $\I(F)$ and for each $S \in \I(F)$, compute $[S]$. In a recent paper \cite{arbol}, the authors address the first part, that is, it is given an efficient procedure to compute $\I(F)$. In Section \ref{sec:4} we center on describing an algorithm that allows to compute $[S]$ when $S \in \I(F)$ is given. Finally, in Section \ref{sec:5} we translate the results in the other sections in terms of the Kunz-coordinates vectors with respect to $F+1$ to provide an efficient algorithm to compute $[S]$, and thus $\Sem(F)$.

\section{A partition of $\Sem(F)$}
\label{sec:2}

In this section we describe a partition of the elements in $\Sem(F)$, for some positive integer $F$, based on the congruence induced by a semigroup homomorphism. It leads us to a simple methodology to enumerate the elements in $\Sem(F)$ by analyzing each of congruence classes that define the partition.

Through this paper, the power set is denoted by $\P(X) = \{A: A \subseteq X\}$, for any set $X$.  For integers $a$ and $b$, we say that $a$ divides $b$ if there exists an integer $c$ such
that $b = ca$, and we denote this by $a | b$. Otherwise, $a$ does not divide $b$, and
we denote this by $a\nmid b$.

Let $F$ be a positive integer, we denote by $\Nrm(F) = \{ n \in \N \backslash\{0\}: n < \frac{F}{2} \text{ and } n\nmid F\}$. It is clear that $(\P(\Nrm(F)), \cap)$ is a semilattice.

\begin{lem}
\label{lem:1}
Let $F$ be a positive integer, then the correspondence $\theta: \Sem(F) \longrightarrow \P(\Nrm(F))$ defined as $\theta(S):=\{s \in S\backslash\{0\}: s<\frac{F}{2}\}$ is a semigroup homomorphism.
\end{lem}

\begin{proof}
Let us see first that $\theta$ is an application. We need to prove that $\{s \in S\backslash\{0\}: s<\frac{F}{2}\} \subseteq \Nrm(F)$. It is clear since if $s | F$, then $F \in S$ which is not possible.

To conclude the proof, the reader can easily check that for $S_1, S_2 \in \Sem(F)$, $\theta(S_1 \cap S_2) = \theta(S_1) \cap \theta(S_2)$.
\end{proof}

Let $R$ be the kernel congruence associated to $\theta$ ($S R S'$ if $\theta(S)=\theta(S')$). For $S \in \Sem(F)$ we denote by $[S]=\{S' \in \Sem(F): SRS'\}$. Then, the quotient set $\frac{\Sem(F)}{R} = \{[S]: S \in \Sem(F)\}$ is also a semilattice with the operation $[S_1] \cap [S_2] = [S_1 \cap S_2]$.

\begin{lem}
\label{lem:2}
Let $S \in \Sem(F)$. Then, $[S]$ is a subset of $\Sem(F)$ closed under union and intersection.
\end{lem}

\begin{proof}
Let $S_1, S_2 \in [S]$. Then, $\theta(S_1)=\theta(S_2) = \theta(S)$. By applying Lemma \ref{lem:1}, we have that $\theta(S_1 \cap S_2) = \theta(S_1) \cap \theta(S_2) = \theta(S)$. Hence, $S_1 \cap S_2 \in [S]$.

Let us see now that $S_1 \cup S_2 \in [S]$. First, we prove that $S_1 \cup S_2 \in \Sem(F)$. It is enough to see that $S_1\cup S_2$ is a semigroup since it is clear that in such a case, $\F(S_1 \cup S_2) = F$. Let us see then that the addition of two elements in $S_1 \cup S_2$ is an element in $S_1 \cup S_2$. Since $S_1$ and $S_2$ are semigroups, we only need to see that if $n \in S_1\backslash S_2$ and $m\in S_2 \backslash S_1$, then $n+m \in S_1 \cup S_2$. We know that $\theta(S_1) = \theta(S_2)$, so we deduce that $n>\frac{F}{2}$ and $m > \frac{F}{2}$. Then, $n+m >F$ and $n+m \in S_1 \cup S_2$. Finally, it is clear that $\theta(S_1 \cup S_2)= \theta(S)$, so $S_1 \cup S_2 \in [S]$.
\end{proof}

A \emph{lattice} is a set with two associative, commutative idempotent binary operations linked by corresponding absorption laws. If the two distributive laws are also verified (with respect to the two operations) the lattice is called distributive. If both operations has a neutral element we say that the lattice has a neutral element.

\begin{theo}
\label{theo:3}
Let $F$ be a positive integer, then $\left( \frac{\Sem(F)}{R}, \cap\right)$ is isomorphic to a subsemilattice of $\left(\P(\Nrm(F)), \cap\right)$. Moreover, $\frac{\Sem(F)}{R}$ is partition of $\Sem(F)$ in distributive lattices with neutral elements.
\end{theo}

\begin{proof}
Let $\theta: \Sem(F) \longrightarrow \P(\Nrm(F))$ the semigroup homomorphism defined in Lemma \ref{lem:1}. By elemental theory of semigroups, we know that ${\rm Im}(\theta) = \{\theta(S): S \in \Sem(F)\}$ is a subsemilattice of $\P(\Nrm(F))$ and that the semilattices $\left( \frac{\Sem(F)}{R}, \cap\right)$ and $({\rm Im}(\theta), \cap)$ are isomorphic.

From Lemma \ref{lem:2} we deduce that $\frac{\Sem(F)}{R}$ is a partition of $\Sem(F)$ into distributive lattices. To conclude the proof note that if $[S] \in \frac{\Sem(F)}{R}$, then $\Zrm([S])=\bigcap_{S' \in [S]} S'$ is the neutral element of $([S], \cup)$ and $\Urm([S])=\bigcup_{S' \in [S]} S'$ is the neutral element of $([S], \cap)$.
\end{proof}

Observe that $\Zrm([S])=\bigcap_{S' \in [S]} S'$ (resp. $\Urm([S])=\bigcup_{S' \in [S]} S'$) is the minimum (resp. maximum) of $[S]$ with respect to the inclusion ordering.
\section{The neutral elements in $[S]$}
\label{sec:3}
In what follows we analyze the neutral elements of each of the operations that gives to $[S]$ the structure of lattice, for any $S \in \Sem(F)$. We see that the set of irreducible and homogeneous numerical semigroups play an important role in this study.

Let $F$ be a positive integer. Recall that $\I(F)$ denotes the set of irreducible numerical semigroups with Frobenius number $F$. Our first goal in this section is to prove that $\I(F) = \{\Urm([S]): [S] \in  \frac{\Sem(F)}{R}\}$.

The following result is deduced in \cite{springer}:
\begin{lem}
\label{lem:4}
Let $S$ be a numerical semigroup with Frobenius number $F$. Then:
\begin{enumerate}
\item\label{lem4:1} $S$ is irreducible if and only if $S$ is maximal (with respect to the inclusion ordering) in $\Sem(F)$.
\item\label{lem4:2} If $h=\max \{x \in \N\backslash S: F-x \not\in S \text{ and } x \neq \frac{F}{2}\}$, then $S \cup \{h\} \in \Sem(F)$.
\item\label{lem4:3} $S$ is irreducible if and only if $\{x \in \N\backslash S: F-x \not\in S \text{ and } x \neq \frac{F}{2}\} \neq \emptyset$.
\end{enumerate}
\end{lem}

From \eqref{lem4:2} and \eqref{lem4:3} in the lemma above we deduce the following result.

\begin{lem}
\label{lem:5}
Let $S \in \Sem(F)$, then, $S \cup \{x \in \N\backslash S: F-x \not\in S \text{ and } x > \frac{F}{2}\} \in \I(F)$.
\end{lem}

We are now ready to prove the result announced at the beginning of this section.

\begin{prop}
\label{prop:6}
Let $S \in \Sem(F)$, then $[S] \cap \I(F) = \{\Urm([S])\}$. Moreover, $\Urm([S]) = S \cup \{x \in \N\backslash S: F-x \not\in S \text{ and } x > \frac{F}{2}\}$.
\end{prop}

\begin{proof}
Let us see first that $\Urm([S]) \in \I(F)$. Suppose that $\Urm([S])$ is not irreducible, then by \eqref{lem4:2} and \eqref{lem4:3} in Lemma \ref{lem:4}, we have that there exists $h=\max \{x \in \N\backslash \Urm([S]): F-x \not\in \Urm([S]) \text{ and } x \neq \frac{F}{2}\}$. Thus, $\Urm([S]) \cup \{h\} \in \Sem(F)$. It is clear also that $h > \frac{F}{2}$ and then $\Urm([S]) \cup \{h\} \in [S]$, contradicting the maximality of $\Urm([S])$ in $[S]$.

Let us see now that if $S' \in [S] \cap \I(F)$ then $S' = \Urm([S])$. Observe that $S'$ and $\Urm([S])$ are two irreducible numerical semigroups with Frobenius number $F$ and $S' \subseteq \Urm([S])$. By applying \eqref{lem4:1} in Lemma \ref{lem:4}, we get that $S'=\Urm([S])$.

Finally, as a direct consequence of Lemma \ref{lem:5} we have that $S \cup \{x \in \N\backslash S: F-x \not\in S \text{ and } x > \frac{F}{2}\} \in \I(F) \cap [S] = \{\Urm([S]\}$.
\end{proof}

Recall the for any positive integer $F$, $\Ho(F)$ denotes the set of homogeneous numerical semigroups with Frobenius number $F$. Our next goal in this section is to prove that $\Ho(F) = \{\Zrm([S]): [S] \in \frac{\Sem(F)}{R}\}$.

The following result has an easy proof and appears in \cite{springer}.

\begin{lem}
\label{lem:7}
Let $S$ be a numerical semigroup and $x$ a minimal generator of $S$. Then, $S \backslash \{x\}$ is also a numerical semigroup.
\end{lem}

Let $a \in \N$. We use $\{a, \rightarrow\}$ to denote the set $\{z \in \N: z \geq a\}$.

\begin{prop}
\label{prop:8}
Let $S \in \Sem(F)$, then $[S]\cap \Ho(F) = \{\Zrm([S])\}$. Moreover, $\Zrm([S])= \langle \theta(S) \rangle \cup \{F+1, \rightarrow\}$.
\end{prop}
\begin{proof}
Let us see first that $\Zrm([S]) \in  \Ho(F)$. If $\Zrm([S])$ is not homogeneous, then $\Zrm([S])$ has a minimal generator $x$ in $]\frac{F}{2}, F[$. By applying Lemma \ref{lem:7} we deduce that $\Zrm([S]) \backslash \{x\} \in [S]$, contradicting the minimality of $\Zrm([S])$.

Let us see now that if $S' \in [S] \cap \Ho(F)$, then $S'=\Zrm([S])$. Since $S' \in [S]$, then $\Zrm([S]) \subseteq S'$. We prove now the other inclusion. Let $x \in S'$. If $x < \frac{F}{2}$ then, $x \in \Zrm([S])$ since $S' R \Zrm([S])$. Because $S'$ has not minimal generators in $]\frac{F}{2}, F[$, we deduce that if $\frac{F}{2} <x < F$ then $x \in \Zrm([S])$. Finally, if $x>F$, then $x \in \Zrm([S])$ since $\F(\Zrm([S]))=F$.

To conclude the proof of the proposition note that $\langle \theta(S) \rangle  \cup \{F+1, \rightarrow\} \in \Ho(F) \cap [S]$. Hence, $S' = \Zrm([S])$.
\end{proof}

As a direct consequence of Propositions \ref{prop:6} and \ref{prop:8} we have the following result.

\begin{cor}
\label{cor:9}
The correspondence $\Delta: \I(F) \rightarrow \Ho(F)$ defined by $\Delta(S) = \langle \theta(S) \rangle  \cup \{F+1, \rightarrow\}$ is a bijective application. Furthermore, $\Delta^{-1}: \Ho(F) \rightarrow \I(F)$ is given by $\Delta^{-1}(S) = S \cup \{x \in \N\backslash S: F-x \not\in S \text{ and } x > \frac{F}{2}$.
\end{cor}

Another immediate corollary that we get from the above results is the following:
\begin{cor}
\label{cor:10}
Let $S \in \I(F)$. Then $[S] = \{S\}$ if an only if $S$ is homogeneous.
\end{cor}
\section{An algorithm to compute $\Sem(F)$}
\label{sec:4}
The goal of this section is to describe an algorithmic procedure to compute $\Sem(F)$ by enumerating each of the elements of the congruence classes in $\frac{\Sem(F)}{R}$.

Observe that as a consequence of Proposition \ref{prop:6} we have the following result.

\begin{lem}
\label{lem:10}
Let $F$ be a positive integer, then $\Sem(F)=\bigcup_{S \in \I(F)} [S]$. Moreover, if $S, S' \in \I(F)$ and $S\neq S'$, then $[S] \cap [S'] = \emptyset$.
\end{lem}

In \cite{arbol} it is shown an algorithmic procedure to compute all the elements in $\I(F)$. Hence, to compute $\Sem(F)$ we concentrate on describe a procedure that computes $[S]$ from a given $S \in \I(F)$.

Let $S \in  \I(F)$. We denote by $\Delta(S)=\Zrm([S])$. Observe that $S=\Urm([S])$. The next result has an immediate proof.

\begin{lem}
\label{lem:11}
Let $S \in \I(F)$ and $S' \in \Sem(F)$. Then $S' \in [S]$ if and only if $\Delta(S) \subseteq S' \subseteq S$.
\end{lem}

If $S \in \I(F)$ we denote by $\D(S)=S \backslash \Delta(S)$. Note that $\Delta(S)= \langle \theta(S) \rangle \cup \{F+1, \rightarrow\}$. Thus, $\D(S)$ is easy to compute from $S$.

For any two sets of nonnegative integers $A$ and $B$ we denote by $A+B= \{a+b: a \in A, b \in B\}$.

\begin{lem}
\label{lem:12}
Let $S \in \I(F)$ and let $B$ be a subset of $\D(S)$. Then, $\Delta(S)$ $\cup$ \\$\left( (B+\Delta(S)) \cap \D(S)\right) \in [S]$. Moreover, all the elements in $[S]$ are in that form.
\end{lem}

\begin{proof}
Let $\overline{S} = \Delta(S) \cup \left( (B+\Delta(S)) \cap \D(S)\right)$. It is clear that $\Delta(S) \subseteq \overline{S} \subseteq S$. Then, by Lemma \ref{lem:11} to prove that $\overline{S} \in [S]$ it is enough to see that $\overline{S}$ is a semigroup. For that, we only need to prove that the addition of two elements in $B$ belongs to $\overline{S}$. This is clear since $B \subseteq \D(S)$ and then, all the elements in $B$ are greater than $\frac{F}{2}$. Consequently, the addition of two elements in $B$ is greater than $F$ and then, belonging to $\Delta(S)$.

Let $S' \in [S]$. From Lemma \ref{lem:11} we deduce that $S' = \Delta(S) \cup B$ with $B \subseteq \D(S)$. By applying that $S'$ is a semigroup we get that $S' =\Delta(S) \cup \left( (B+\Delta(S)) \cap \D(S)\right)$.
\end{proof}

Let $d \in \D(S)$. We denote by $\T(d) = \left( \{d\} + \Delta(S)\right) \cap \D(S)$. If $B \subseteq \D(S)$, then $\T(B) = \bigcup_{b \in B} \T(b)$. The next result is a reformulation of Lemma \ref{lem:12} with this new notation.

\begin{prop}
\label{prop:13}
Let $S \in \I(F)$ and $A=\{\T(B): B \subseteq \D(S)\}$. Then, $[S] = \{\Delta(S) \cup X: X \in A\}$.
\end{prop}

The pseudo-code in Algorithm \ref{alg:1} shows how to compute $[S]$ for any $S\in \I(F)$.

\begin{algorithm2e}[h]

\caption{Computation of the class of $S$.\label{alg:1}}

\SetKwInOut{Input}{Input}
\SetKwInOut{Output}{Output}
\Input{$S \in  \I(F)$.}

$\bullet$ Compute $\Delta(S) = \langle \theta(S) \rangle \cup \{F+1, \rightarrow\}$ and $\D(S) = S \backslash \Delta(S)$.

$\bullet$ Set $A= \{\T(B): B \subseteq \D(S)\}$.

\Output{$[S] = \{\Delta(S) \cup X: X \subseteq A\}$.}
\end{algorithm2e}

In the following example we illustrate the usage of Algorithm \ref{alg:1}.
\begin{ex}
\label{ex:1}
Let us compute $[S]$ for $S=\langle 3, 5\rangle \in \I(7)$. Observe that $\theta(S)=\{3\}$.
\begin{itemize}
\item $\Delta(S) = \langle 3 \rangle \cup \{8, \rightarrow\} = \langle 3,8,10 \rangle$ and $\D(S) = \{5\}$.
\item $\T(5) = \left( \{5\} + \langle 3,8,10 \rangle\right) \cap  \{5\} = \{5\}$. $\T(\{5\})=\{5\}$  $A = \{\emptyset, \{5\}\}$.
\end{itemize}
Then, $[S] = \{\langle 3,8,10 \rangle \cup \emptyset, \langle 3,8,10 \rangle \cup \{5\}\} = \{ \langle 3,8,10 \rangle, \langle 3, 5 \rangle\}$.
\end{ex}

\section{The Kunz-coordinates vectors of  $\Sem(F)$}
\label{sec:5}

In this section we use a different encoding of a numerical semigroups to compute $[S]$ for any $S\in \I(F)$, and for any positive integer $F$. We present an analogous construction to the one in the section above but based on manipulating vectors in $\{0,1\}^F$. It leads us to an efficient procedure to compute the set $\Sem(F)$.

 Let $S$ be a numerical semigroup and $n \in S\backslash \{0\}$. The \emph{Ap\'ery set} of $S$ with respect to $n$ is the set $\Ap(S,n) = \{s
\in S :  s - n \not\in S\}$. This set was introduced by Ap\'ery in \cite{apery}.

The following characterization  of the Ap\'ery set that appears in \cite{springer} will be useful for our development.
 \begin{lem}
 \label{lem:13}
Let $S$ be a numerical semigroup and $n \in S\backslash \{0\}$. Then $\Ap(S,n) = \{0 = w_0, w_1, \ldots,  w_{n -
1}\}$, where $w_i$ is the least
element in $S$ congruent with $i$ modulo $n$, for $i=1, \ldots, n-1$.
 \end{lem}

Moreover, the set $\Ap(S,n)$ completely determines $S$, since $S =
\langle \Ap(S,n) \cup \{n\} \rangle$ (see \cite{london}), and then, we can identify $S$ with its Ap\'ery set with respect to $n$. The set $\Ap(S,n)$ contains, in general, more information than an arbitrary system of
generators of $S$. For instance, Selmer in \cite{selmer77} gives the formulas,
$\g(S)=\frac{1}{n}\left(\sum_{w \in \Ap(S, n)} w\right) -
\frac{n-1}{2}$ and $\F(S) = \max(\Ap(S,n)) - n$. One can also test if a nonnegative integer $s$ belongs to $S$ by checking if $w_{s\pmod m} \leq s$.

We consider an useful modification of the Ap\'ery set that we call the \emph{Kunz-coordinates vector} as in \cite{bp2011}. Let $S$ be a numerical semigroup and $n \in S\backslash\{0\}$. By Lemma \ref{lem:13}, $\Ap(S, n)=\{w_0=0, w_1, \ldots, w_{n-1}\}$, with $w_i$ congruent with $i$ modulo $n$. The \emph{Kunz-coordinates vector} of $S$ with respect to $n$ is the vector $\Kunz(S,n) = x \in \N^{n-1}$ with components $x_i = \frac{w_i-i}{n}$ for $i=1, \ldots, n-1$. If $x \in \N^{n-1}$ is a Kunz-coordinates vector, we denote by $S_x$ the numerical semigroup such that $\Kunz(S_x, n) = x$.
The Kunz-coordinates vectors were introduced in \cite{kunz} and have been previously analyzed when $n=\m(S)$ in \cite{bp2011,london}.

We also denote by $\Kunz(F) = \{\Kunz(S, F+1) : S \in \Sem(F)\}$ the set of Kunz-coordinates vectors associated to the numerical semigroups with Frobenius number $F$ and by $\Kunz^\I (F) = \{\Kunz(S, F+1) : S \in \I(F)\}$ (resp. $\Kunz^\Ho (F) = \{\Kunz(S, F+1) : S \in \Ho(F)\}$) the set of Kunz-coordinates vectors of the set of irreducible (resp. homogeneous) numerical semigroups with Frobenius number $F$.

For $x, y \in \Kunz(F)$ we define the binary operation
$$
x \bullet y = (\max\{x_1, y_1\}, \ldots, \max\{x_F, y_F\}).
$$
It is easy to see that if $x, y \in \Kunz(F)$, $\Kunz(S_x \cap S_y, F+1) = \Kunz(S_x, F+1)\,\bullet\,\Kunz(S_y, F+1)$ and then, $(\Kunz(F), \bullet)$ is a semilattice.

It is also clear that $\Kunz(\cdot, F+1): \Sem(F) \longleftarrow \Kunz(F)$, is a semigroup isomorphism with inverse $\varphi(x) = \langle F+1, (F+1)x_1+1, \ldots, (F+1)x_F+F\rangle$, for any $x \in \Kunz(F)$. We define $\theta^K$ as the unique homomorphism such that the following diagram holds:
$$
\xymatrix{ \Sem(F) \ar[rr]^{\Kunz(\cdot, F+1)} \ar[dr]_{\theta}& & \Kunz(F) \ar[dl]^{\theta^K} \\
& \P(\Nrm(F)) &}
$$
Then, we define the kernel congruence over $\Kunz(F)$, $R^K$, associated to $\theta^K$. Thus, $\frac{\Kunz(F)}{R^K}$ is also a semilattice with the operation $[x]\bullet[y] = [x\bullet y]$. By Lemma \ref{lem:2} we have that $[x]$ is closed under the operation $\bullet$ and also under the operation $a \star b = (\min\{a_1, b_1\}, \ldots, \min\{a_F, b_F\})$ (which coincides with the scalar product of $a$ and $b$), for any $a, b \in [x]$. Note that if $x, y \in \Kunz(F)$, $S_{x\bullet y} = S_x \cap S_y$ and $S_{x \star y} = S_x \cup S_y$.

By Theorem \ref{theo:3} we have that $\frac{\Kunz(F)}{R^K}$ is a partition of $\Kunz(F)$ in distributive lattices with neutral elements. Actually, this neutral elements are $\Urm^K([x]) = (\max\{y_1: y \in [x]\}, \ldots, \max\{y_F: y \in [x]\})$ (for $\bullet$) and $\Zrm^K([x]) = (\min\{y_1: y \in [x]\}, \ldots, \min\{y_F: y \in [x]\})$ (for $\star$).

In the following result it is shown the structure of the Kunz-coordinates vector of a numerical semigroup with Frobenius number $F$.
\begin{lem}
\label{lem:14}
Let $S$ be a numerical semigroup and $x=\Kunz(S, \F(S)+1) \in \N^F$. Then:
\begin{enumerate}
\item $ x \in \{0,1\}^{\F(S)}$,
\item\label{lem14:2} $\G(S) = \{i \in \{1, \ldots, \F(S)\}: x_i =1\}$,
\item $\g(S)= \dsum_{i=1}^F x_i$.
\item $i \in \{1, \ldots, \F(S)\}$ is a minimal generator of $S$ if and only if $x_i=0$ and $x_j + x_{i-j} \geq 1$ for all $j<i$.
\end{enumerate}
\end{lem}
It is not difficult to see, from \cite{arbol} or \cite{london}, that $\Kunz(F) = \{x \in \{0,1\}^F: x_F=1 \mbox{ and } x_i + x_j - x_{i+j} \geq 0, \mbox{ for all } 1 \leq i \leq j \leq F \mbox{ with } i+j \leq F\}$.

Furthermore, from \eqref{lem14:2} in Lemma \ref{lem:14} we have that $\theta^K(x) = \{i \in \{1, \ldots, \left\lceil \frac{F}{2}\right\rceil -1\}: x_i=0\}$, for all $x \in  \Kunz(F)$. (Here, we denote by $\lceil z \rceil  = \min\{n \in \Z: z\leq n\}$ the ceiling part of any rational number $z$.)

It is well-known (see for instance Corollary 3.5 in \cite{springer}) that a numerical semigroup $S$ is irreducible if and only if $\g(S) = \lceil \frac{\F(S)+1}{2}\rceil$.

The following result, whose proof is direct, allows us to characterize the Kunz-coordinates vectors of the elements in $\I(F)$ and $\Ho(F)$.
\begin{lem}
\label{lem:15}
Let $F$ be a positive integer. Then, $\Kunz^\I(F)$ is the set of solutions of the following system of binary inequalities and equations:
\begin{align}
x_i + x_j - x_{i+j} &\geq 0  \mbox{, for all $1 \leq i \leq j \leq F$ with $i+j \leq F$}\nonumber\\
\dsum_{i=1}^F x_i &= \left\lceil \frac{F+1}{2} \right\rceil\nonumber\\
x_F&=1\nonumber\\
x &\in \{0,1\}^F\nonumber
\end{align}
And $\Kunz^\Ho(F)$ is the set of solutions of the following system:
\begin{align}
x_i + x_j - x_{i+j} &\geq 0  \mbox{, for all $1 \leq i \leq j \leq F$ with $i+j \leq F$}\nonumber\\
x_j + x_{i-j} &\leq 2x_{i} \mbox{, for all $i > \frac{F}{2}$ and $j < i$}\nonumber\\
x_F&=1\nonumber\\
x &\in \{0,1\}^F\nonumber
\end{align}
\end{lem}
From Propositions \ref{prop:6} and \ref{prop:8} we have that for any $x\in \Kunz(F)$, $[x] \cap \Kunz^\I (F) = \{\Urm^K([x])\}$ and $[x] \cap \Kunz^\Ho(F) = \{\Zrm^K([x])\}$. Moreover, $\Kunz(\Delta(S), F+1)$ can be computed from $\Kunz(S,F+1)$ by solving an integer programming problem as stated in the next result.
\begin{prop}
\label{prop:17}
Let $x\in \Kunz^\I(F)$. Then, $\Kunz(\Delta(S_x), F+1)$ is the unique optimal solution of the following binary programming problem:
\begin{align}
& \max \dsum_{i=1}^F z_i& \nonumber\\
s.t.&\nonumber\\
z_i + z_j - z_{i+j} &\geq 0  \mbox{, for all $1 \leq i \leq j \leq F$ with $i+j \leq F$,}\nonumber\\
z_i&=0 \mbox{, for all $i \in \theta^K(x)$,}\tag{${\rm IP}^\Delta (x)$}\label{eq:IPHo}\\
z_F&=1\nonumber\\
z &\in \{0,1\}^F\nonumber
\end{align}
\end{prop}
\begin{proof}
It is suffices to notice that the feasible region of \eqref{eq:IPHo} corresponds with the set of numerical semigroups with Frobenius number $F$  and such that all the elements in $\theta(S_x)$ belong to it. When maximizing the overall sum of the coordinates we get the unique (see Proposition \ref{prop:8}) numerical semigroup fulfilling those conditions with minimum number of elements (or equivalently, the maximum number of gaps), which is $\Delta(S_x)$.
\end{proof}

For $x\in \Kunz^\I(F)$, we denote by $\Delta^K(x)$ the optimal solution of problem \eqref{eq:IPHo} or equivalently, $\Delta^K(x)=\Kunz(\Delta(S_x), F+1)$.

The following results are the translations of the computations in Algorithm \ref{alg:1} needed to enumerate the $[S]$ for $S \in \I(F)$ in terms of their Kunz-coordinates vectors.
\begin{lem}
\label{lem:18}
\label{lem:}
Let $x\in\Kunz^\I(F)$ and $y=\Delta^K(x)$. Then, $\D(S_x) = \{i \in \{1, \ldots, F\}: x_i=0 \text{ and } y_i=1\}$.
\end{lem}

For $i \in \{1, \ldots, F\}$, we denote by $\e_i$ the $F$-tuple having a $1$ as its $i$th entry and zeros otherwise.

\begin{lem}
\label{lem:19}
Let $x\in\Kunz^\I(F)$. If $B \subseteq \D(S_x)$ and $X \subseteq \T(B)$. Then, $\Kunz(\Delta(S_x) \cup X, F+1) = \Delta^K(x) - \dsum_{i \in X} \e_i$.
\end{lem}

The translation of Corollary \ref{cor:10} for Kunz-coordinates vectors is the following:
\begin{cor}
\label{cor:14}
Let $x\in\Kunz^\I(F)$. Then, $[x] = \{x\}$ if and only if for all $i>\frac{F}{2}$ either $x_i=1$ or there exists $j<i$ such that $x_j=x_{i-j}=0$.
\end{cor}

The pseudo-code of the procedure to compute $[S]$ for $S \in \I(F)$ in terms of its Kunz-coordinates vector is shown in Algorithm \ref{alg:2}

\begin{algorithm2e}[H]
\caption{Computation of the class of $x$.\label{alg:2}}

\SetKwInOut{Input}{Input}
\SetKwInOut{Output}{Output}
\Input{$x \in \Kunz^\I(F)$.}

$\bullet$ Solve $\eqref{eq:IPHo}$ and set $y$ as its optimal solution.

$\bullet$ Compute $\D(S_x) = \{i \in \{1, \ldots, F\}: x_i=0 \text{ and } y_i=1\}$.

$\bullet$ Set $A= \{\T(B): B \subseteq \D(S_x)\}$.

\Output{$[x] = \{y - \dsum_{i\in X} \e_i: X \subseteq A\}$.}
\end{algorithm2e}

In the following example we illustrate the proposed methodology to compute the whole set of numerical semigroups with a fixed Frobenius number.
\begin{ex}
\label{ex:20}
We compute the set of numerical semigroups with Frobenius number $5$. By applying the algorithm proposed in \cite{arbol} we get that $\Kunz(\I(F), F+1)=\{ (1, 0, 1, 0, 1), (1, 1, 0, 0, 1)\}$. For each of these two elements we compute its class by Algorithm \ref{alg:2}:
\begin{enumerate}
 \item $x^1=(1, 0, 1, 0, 1)$:
\begin{itemize}
 \item The solution of \eqref{eq:IPHo} is $y^1=(1,0,1,0,1)$. Since $x^1=y^1$, we can conclude the the unique Kunz-coordinates vector in its class is itself since the maximum and the minimum element coincides in this case.
\end{itemize}
Then  $[x^1] = \{(1,0,1,0,1)\}$.
\item $x^2=(1,1,0,0,1)$:
\begin{itemize}
 \item The solution of \eqref{eq:IPHo} is, in this case, $y=(1,1,1,1,1)$.
\item $\D(S_{x^2}) = \{3,4\}$.
\item $A = \{\emptyset,\{3\}, \{4\}, \{3,4\}\}$, since $\T(3)=\{3\}$ and $\T(4)=\{4\}$.
\end{itemize}
Hence, $[x^2] = \{(1,1,1,1,1), (1,1,0,1,1), (1,1,1,0,1), (1,1,0,0,1)\}$.
\end{enumerate}
Then,
$$
\Kunz(5) = [x^1] \cup [x^2] =  \{(1,0,1,0,1), (1,1,1,1,1), (1,1,0,1,1), (1,1,1,0,1), (1,1,0,0,1)\}.
$$
In terms of the semigroup, it is easy to compute from that set that
$$
\Sem(5) = \{\langle 2, 7\rangle,  \langle 6, 7, 8, 9, 10, 11 \rangle, \langle  3, 7, 8 \rangle, \langle  4, 6, 7, 9 \rangle, \langle 3, 4 \rangle
 \}.
$$
\end{ex}

Note that the computations in Algorithm \ref{alg:2} are much faster than those in Algorithm \ref{alg:1} due to the simplicity of working with $0-1$ vectors instead of numerical semigroups. Observe also that, in general, the computation of the Ap\'ery sets and the Kunz-coordinates vectors of a numerical semigroup is hard. However, it is not needed to compute the Ap\'ery set of any numerical semigroup involved in the algorithm, since $\Kunz^\I(F)$ is computed by using the algorithm provided in \cite{arbol} (which works only by manipulating $0-1$ vectors) and $\Delta^K(x)$ is computed without knowing $\Delta(S)$ by solving an integer programming problem. Furthermore, a system of generators of the obtained semigroups is easy to give for any element  $y \in [x]$ by applying that $S_y = \langle F+1, (F+1)y_1+1, \ldots, (F+1)y_F+F\rangle$.

Observe also that some of the computations can be avoided by applying Corollary \ref{cor:14} since in case the hypotheses are fulfilled, neither the integer programming problem nor the other operations are needed.



\begin{thebibliography}{99}
\bibitem{apery} Ap\'ery, R. (1946). Sur les branches superlin\'eaires des courbes alg\'ebriques. C. R. Acad. Sci. Paris 222, 1198--2000.
\bibitem{barucci} Barucci, V., Dobbs, D.E., and Fontana, M. (1997). Maximality properties in numerical semigroups and applications to one-dimensional analitically irreducible local domains.
Memoirs of the American Mathematical Society. Vol.125, n.598.
\bibitem{bp2011} Blanco, V. and Puerto, J. (2011). \emph{Integer programming for decomposing numerical semigroups into $m$-irreducible numerical semigroups}. Submitted. Available in \url{http://arxiv.org/abs/1101.4112}.
\bibitem{arbol} Blanco, V. and Rosales, J.C. \emph{The tree of irreducible numerical semigroups with fixed Frobenius number}. Available in \url{http://arxiv.org/abs/1105.2147}.
\bibitem{froberg} Fr\"oberg, R., Gottlieb, C., and H\"aggkvist, R. On numerical semigroups, Semigroup Forum 35 (1987)
63–83.

\bibitem{kunz} Kunz, E. (1987). \"Uber dir Klassifikation numerischer Halbgruppen,
Regensburger matematische schriften 11.

    \bibitem{london} Rosales, J. C., Garc\'ia-S\'anchez, P. A., Garc\'ia-Garc\'ia, J. I.,
and Branco, M. B. (2002). \emph{Systems of inequalities and numerical semigroups}. J. London Math. Soc. (2) 65, no. 3, 611--623.

\bibitem{pacific}  Rosales, J.C, and Branco, M.B. (2003). Irreducible numerical semigroups, Pacific J. Math. 209 (2003),
 131-143.

\bibitem{springer}
 Rosales, J.C. and Garc\'ia-Sanchez, P.A. (2009). Numerical semigroups, Springer, New York, NY, 2009. ISBN: 978-1-4419-0159-0.

    \bibitem{selmer77} Selmer, E.S. (1977). \emph{On a linear Diophantine problem of Frobenius}, J. Reine Angew. Math.
293/294, 1-17.
\end{thebibliography}
\end{document}